\documentclass[english,11pt,two]{article}

\usepackage[a4paper, margin=2.5cm]{geometry}
\usepackage{epstopdf}
\usepackage{tikz}
\usetikzlibrary{calc}
\usepackage{graphicx}
\usepackage{pgfplots}
\usepackage{bbding}
\usepackage{amsmath}
\usepackage{amssymb}
\usepackage{amsthm}
\usepackage{ifthen}
\usepackage{url}
\usepackage{array}

\newtheorem{theorem}{Theorem}

\newtheorem{proposition}{Proposition}
\newtheorem{corollary}{Corollary}

\newtheorem{assumption}{Assumption}
\newtheorem{remark}{Remark}
\newtheorem{example}{Example}

\newtheorem{definition}{Definition}

\usepackage[font=small,labelfont=bf]{caption}

\renewcommand{\phi}{\ensuremath{\varphi}}

\usepackage{natbib}
 \bibpunct[, ]{(}{)}{,}{a}{}{,}%
 \def\newblock{\ }%

\title{Piecewise constant decision rules via branch-and-bound based scenario detection for integer adjustable robust optimization}
\author{Ward Romeijnders \thanks{Department of Operations, University of Groningen, The Netherlands e-mail: w.romeijnders@rug.nl} 
\and Krzysztof Postek \thanks{Econometric Institute, Erasmus University Rotterdam, The Netherlands, e-mail: postek@ese.eur.nl}}
\date{}
\begin{document}
\maketitle

\abstract{Multi-stage problems with uncertain parameters and integer decisions variables are among the most difficult applications of robust optimization (RO). The challenge in these problems is to find optimal here-and-now decisions, taking into account that the wait-and-see decisions have to adapt to the revealed values of the uncertain parameters. \cite{Postek2016} and \cite{Bertsimas2016} propose to solve these problems by constructing piecewise constant decision rules by adaptively partitioning the uncertainty set. The partitions of this set are iteratively updated by separating so-called critical scenarios of the uncertain parameters. Both references present methods for identifying these critical scenarios. However, these methods are most suitable for problems with continuous decision variables and many uncertain constraints. In particular, they are not able to identify informative sets of critical scenarios for integer problems with uncertainty in the objective function only. In this note, we address this shortcoming of existing methods by introducing a general critical scenario detection method for mixed-integer adjustable robust optimization problems that is based on the branch-and-bound tree used to solve the corresponding static problem. Numerical experiments on a route planning problem show that our general-purpose method outperforms the problem-specific approach of \cite{Postek2016}.}

\section{Introduction}
Robust optimization \citep{BenTal2009} is a paradigm for dealing with uncertainty in mathematical optimization problems where the objective function is minimized under the assumption that the uncertain parameters attain their worst-case value from an uncertainty set, \emph{i.e.}, a set of likely values. This methodology has found a wide range of applications, see, \emph{e.g.}, inventory management \citep{BenTal2004}, facility location \citep{Ordonez2007}, network design \citep{Atamturk2007}, finance \citep{Fabozzi2010}, and many others. For a broad overview of applications of robust optimization (RO), we refer the reader to \citet{Gabrel2014}.

An important class of RO problems are multi-stage problems where here-and-now decisions are implemented before (some of) the uncertain parameters are revealed, and wait-and-see decisions are made when these uncertain parameters are known. The wait-and-see decisions will typically differ for different realizations of the uncertain parameters and this is why we call them adjustable decisions. In general, such adjustable problems are $\mathcal{NP}$-hard \citep{BenTal2009}, even for problems with continuous decision variables only, and thus require good suboptimal but tractable solutions. For this reason, \cite{BenTal2004} propose to formulate the later-stage decisions as affine functions of the uncertain parameters. Their approach has later been extended to other function classes; see, \emph{e.g.}, \citet{Chen2009} and \citet{Bertsimas2011}. An alternative solution method that bypasses the need for decision rules is to use Fourier-Motzkin elimination to remove the later-stage decisions from the problem formulation \citep{Zhen2017}.

Incorporating adjustable decisions in RO problems becomes more challenging if (some of) the decisions are restricted to be integer. In this case, it becomes difficult to formulate these decisions as tractable functions of the uncertain parameters. First attempts to address this difficulty include  \cite{Caramanis2007} who construct rounding-based decision rules based on sampling that are feasible with high probability, and  \citet{Vayanos2011} who partition the uncertainty set \textit{ex ante} into small subsets with different decisions each. 

In the current literature, we distinguish three systematic approaches for designing integer decision rules for mixed-integer adjustable RO problems. The first is to use piecewise linear decision rules for both continuous and binary decision variables, proposed by \citet{Bertsimas2015}. They formulate the decision rules as differences of two convex functions, and for binary variables the value 0 is implemented if the decision rule is positive, and the value 1, otherwise. In a related fashion, the decisions in the approach of \cite{Bertsimas2017} are affine transformations of multiple indicator functions of half-spaces in the space of uncertain parameters.

The second approach is the $K$-adaptability \citep{Caramanis2010}, proposed in the integer context by \citet{Hanasusanto2015}. In this approach, $K$ possible values for the adjustable decisions are selected here-and-now, and for each outcome of the uncertain parameters the best out of these $K$ possible values will be selected for the wait-and-see decisions. The corresponding optimization problem is solved by reformulating it as a static mixed-integer RO problem. This approach was extended by \cite{Subramanyam2017} who allow discrete uncertain parameters and develop a branch-and-bound algorithm for the $K$-adaptable problem. 

The third approach is the splitting methodology proposed by \cite{Postek2016} and \cite{Bertsimas2016}; the latter use the term `partitioning' instead of `splitting'. In this approach, the uncertainty set is iteratively split into smaller subsets. For each subset, a possibly different value for the adjustable decisions is selected that will be implemented if the uncertain parameter turns out to be in that subset. The uncertainty set is split based on \emph{critical scenarios} of the uncertain parameters, since the theory for detecting these critical scenarios shows that if they are not separated from each other, the objective value of the solution induced by the split uncertainty set cannot improve. This theory, however, only holds for problems with continuous decision variables, and can only be heuristically applied to some mixed-integer problems. In particular, for mixed-integer adjustable RO problems with uncertainty in the objective function only, this theory is unable to detect critical scenarios that need to be split.

We address exactly this shortcoming by detecting critical scenarios in mixed-integer adjustable RO problems. Our method is based on the branch-and-bound (B\&B) tree (see, \textit{e.g.}, \cite{Schrijver1986}) used for solving the corresponding static mixed-integer RO problem (where the later-stage decisions are the same for all realizations of uncertainty). In fact, the critical scenarios are obtained directly from the optimal dual solutions of the LP relaxations of a specific set of nodes in this B\&B tree. This means that they can be obtained as by-product when solving the static mixed-integer RO problem. 

In this note, we only present our critical scenario detection method for  \emph{two-stage} mixed-integer adjustable RO problems for ease of exposition. However, similarly as in \cite{Postek2016} and \cite{Bertsimas2016} it can easily be extended to the multi-stage case by enforcing the nonanticipativity constraints.

The main contributions of our note are as follows:
\begin{itemize}
\item we detect critical scenarios in mixed-integer adjustable RO problems using branch-and-bound, extending the theory of \cite{Postek2016} and \cite{Bertsimas2016};
\item we derive an optimality criterion for our splitting methodology, proving when the uncertainty set requires no more partitioning;
\item we show using numerical experiments on a route planning problem that our general-purpose critical scenario detection method outperforms the problem-specific heuristic developed in \cite{Postek2016}.
\end{itemize}


The remainder of this paper is organized as follows. In Section~\ref{sec.Splitting} we review the splitting methodology of \cite{Postek2016} and \cite{Bertsimas2016}. In Section~\ref{sec.BBsplitting} we derive our critical scenario detection method. In Section~\ref{sec.experiment} we illustrate our method using numerical experiments on a route planning problem, and we end with conclusions in Section \ref{sec.Summary}.

\section{Splitting methodology for mixed-integer adjustable RO problems} \label{sec.Splitting}
We consider the mixed-integer adjustable RO problem:
\begin{align}
\bar{t} := \min\limits_{t, x, y(z)} \ & t \label{eq.initial.problem} \tag{ARO} \\
\text{s.t.} \ & t - c(z)^\top x - q(z)^\top y(z) \geq 0 &&  \forall z \in Z \nonumber \\
& a_i(z)^\top x + w_i(z)^\top y(z) \geq b_i && \forall z \in Z && \forall i \in \mathcal{I} \nonumber \\
& x \in X, \,\, y(z) \in Y            && \forall z \in Z, \nonumber
\end{align}
where the uncertainty is in both the cost parameters $c(z), q(z)$ and the constraint coefficients $a_i(z), w_i(z), i \in \mathcal{I}$, with $z$ representing the uncertain parameters in the model and $Z$ a polyhedral uncertainty set defined by $Z = \{z \in \mathbb{R}^L: Pz \leq p\}$, and the sets $X$ and $Y$ represent non-negativity and integer restrictions. In this problem the decisions $x \in \mathbb{R}^{d_1}$ have to be determined \textit{before} the value of the uncertain parameter $z$ is known, whereas decisions $y(z) \in \mathbb{R}^{d_2}$ are made after the realizations of $z$ are revealed. We assume w.l.o.g. that the first $m_1$ and $m_2$ components of the decision vectors $x$ and $y(z)$, respectively, are restricted to be integer. Thus, $X = \mathbb{Z}^{m_1}_{+} \times \mathbb{R}^{d_1-m_1}_{+}$ and $Y = \mathbb{Z}^{m_2}_{+} \times \mathbb{R}^{d_2-m_2}_{+}$. Moreover, we make the following assumptions with respect to the uncertainty set $Z$ and the parameters in the model.
\begin{assumption} \label{assumption.nonempty.bounded}
The uncertainty set $Z$ is nonempty and bounded.
\end{assumption}
\begin{assumption}  \label{assumption.polyhedron}
All parameters $c(z)$, $q(z)$, $a_i(z)$ and $w_i(z)$ are affine in the uncertain parameter $z$. That is, $ c(z) = \overline{c} + Cz, \ q(z) = \overline{q} + Qz, \ a_i(z) = \overline{a}_i + A_i z$, and  $w_i(z) = \overline{w}_i + W_i z$,
where $\overline{c}, \overline{a}_i \in \mathbb{R}^{d_1}$ and $\overline{q}, \overline{w}_i \in \mathbb{R}^{d_2}$ represent the nominal values and $C, A_i \in \mathbb{R}^{d_1 \times L}$ and $Q, W_i \in \mathbb{R}^{d_2 \times L}$.  
\end{assumption} 
\cite{BenTal2009} show that the adjustable optimization problem is NP-hard, even when all decision variables are continuous. For this reason, a typical approach to simplify such problems is to restrict $y(z)$ to a particular class of functions. For example, \cite{BenTal2004} propose so-called affine decision rules, meaning that $y(z)$ is an affine function of $z$. The problem with this approach in our setting is that affine decision rules cannot be applied if some of the second-stage decisions are integer.

Instead, we follow the approach of \cite{Postek2016} and \cite{Bertsimas2016} and construct piecewise constant decision rules for mixed-integer adjustable RO problems by splitting the uncertainty set. After $r$ rounds of iterative splitting we obtain a partition $\mathcal{Z}_r$ of $Z$ given by $\mathcal{Z}_r = \{ Z_{r,s}, \ s \in \mathcal{S}_r \}$ where $Z_{r,s}$ are nonempty subsets of $Z$ with mutually disjoint interiors and such that $\cup_{s \in \mathcal{S}_r } Z_{r,s} = Z$. A piecewise constant decision rule can now be obtained by assuming that for each $s \in \mathcal{S}_r $, we will select the second-stage decision $y^{r,s}$ for each $z \in Z_{r,s}$. That is, for all $z \in Z$,
$$ 
y(z) =  y^{r,s} \ \quad \text{if} \ z \in Z_{r,s}.
$$
Under this assumption, the mixed-integer adjustable RO problem (\ref{eq.initial.problem}) reduces to 
\begin{align} \label{eq:two_period_after_general_splitting} \tag{$\text{RO}_r$}
\bar{t}^r := \min\limits_{t^r, x^r, y^{r,s}} \ & t^r \\
\text{s.t.} \ &t^r - c(z)^Tx^r - q(z)^Ty^{r,s} \geq 0 && \forall z \in Z_{r,s} &&  \forall s \in \mathcal{S}_r   \nonumber \\
& a_i(z)^\top x^r + w_i(z)^\top y^{r,s} \geq b_i && \forall z \in Z_{r,s} && \forall  i \in \mathcal{I}, \ \forall s \in \mathcal{S}_r  \nonumber \\
& x \in X, \,\, y^{r,s} \in Y &&&& \forall s \in \mathcal{S}_r . \nonumber
\end{align}
where $Z_{r,s} = \{z: P^{r,s} z \leq p^{r,s} \}$ for all $s \in \mathcal{S}_r $. That is, we assume that the subsets $Z_{r,s}$ are weakly separated from each other by hyperplanes.

The problem (\ref{eq:two_period_after_general_splitting}) is a static mixed-integer RO problem in which all decisions, \emph{i.e.}, $x^r$ and $y^{r,s}$, $s \in \mathcal{S}_r$, have to be determined before the uncertain parameter $z$ is known. Clearly, by iteratively splitting the uncertainty sets $Z_{r,s}$, the approximation (\ref{eq:two_period_after_general_splitting}) of the mixed-integer adjustable RO problem (\ref{eq.initial.problem}) iteratively improves.

Our contribution is that we determine how to iteratively split the uncertainty sets. This is a generalization of the results of \cite{Postek2016} for the case in which all decision variables are continuous. In this case, they detect a finite set $\overline{Z}_{r,s}$ of \emph{critical scenarios} for each uncertainty subset $Z_{r,s}, \ s \in \mathcal{S}_r$, after each splitting round $r$. They show that if none of these sets $\overline{Z}_{r,s}$ of critical scenarios are split in round $r+1$, then $\bar{t}^{r+1} = \bar{t}^r$, \emph{i.e.}, the worst-case objective value does not decrease. This provides a theoretical justification for splitting the sets $\overline{Z}_{r,s}$ of critical scenarios.

This theoretical result, however, does not hold when some of the decision variables in the model are restricted to be integer. Therefore, \cite{Postek2016} propose to use the critical scenarios $\overline{Z}_{r,s}$ of the LP relaxation of (\ref{eq:two_period_after_general_splitting}). However, in general this approach does not work. For example, if there is only uncertainty in the cost parameters $c(z)$ and $q(z)$, then the LP relaxation may only find a single critical scenario $\overline{z}_{r,s}$ per uncertainty subset $Z_{r,s}$, giving us no indication on how to splits these sets. 

\begin{example} \label{example.relaxation.vs.milp}
In Figure~\ref{fig.detection} we graphically illustrate why it may be insufficient to use only the critical scenarios of the LP relaxation of (\ref{eq:two_period_after_general_splitting}). 

In this example, we assume that there are two adjustable decision variables $y_1$ and $y_2$ both of which are integer. Moreover, we assume that there is no uncertainty in the cost parameters $q(z)$. Thus, only the feasible region of this problem depends on the uncertain parameter $z$. In Figure~\ref{fig.detection} the feasible regions of the LP relaxation of (\ref{eq:two_period_after_general_splitting}), corresponding to two realizations $\overline{z}^1$ and $\overline{z}^2$ of this uncertain parameter, are represented by the two quadrilaterals.

In the left panel, the feasible region of the LP relaxation of (\ref{eq:two_period_after_general_splitting}) is depicted as the shaded intersection of the two quadrilaterals. At the optimal solution $(1.25, 1.5)$, only the constraints corresponding to scenario $\overline{z}^2$ are active, and thus $\overline{z}^2$ is the only critical scenario. Notice that without the constraints corresponding to $\overline{z}^1$, the same solution would be optimal, and thus the worst-case objective function of (\ref{eq.initial.problem}) will not improve if we are allowed to make different (continuous) decisions $y_1$ and $y_2$ for the different scenarios $\overline{z}^1$ and $\overline{z}^2$.

In the right panel, we consider the integer RO problem (\ref{eq:two_period_after_general_splitting}). Its feasible region consists of all integer points in the intersection of the two quadrilaterals with optimal solution $(2,2)$. Contrary to the LP relaxation, it is possible to improve the worst-case objective value of (\ref{eq.initial.problem}) by separating scenarios $\overline{z}^1$ and $\overline{z}^2$ and thus making different decisions $(y_1, y_2)$ for the scenarios -- $(1,2)$ and $(2,1)$ -- both with smaller objective function values than $(2, 2)$.   
\begin{figure}
\centering
\begin{tikzpicture}[scale = 0.8]

\def\circlerad{2.0}

\draw [<->,thick] (0,6) -- (0,0) -- (8,0);
\draw[help lines,opacity = 0.5] (0,0) grid (8,6);
\node at (8, -0.5){{\scriptsize $y_1$}};
\node at (-0.5, 6){{\scriptsize $y_2$}};
\draw [->,thick] (1, 2) -- (1 - 1.5, 2 -0.75);

\foreach \y in {1,...,5}
	\draw (-0.1, \y) -- (0.1, \y);
	
\foreach \y in {1,...,7}
	\draw  (\y, -0.1) -- (\y, 0.1);

\draw [thick, color=black] (1.25,5.75) -- (1.25,1.5) -- (2, 0.75)  -- (6, 5.75) -- (1.25,5.75);
\node at (6.5, 3){{\scriptsize $\overline{z}^1$}};
\draw [thick, color=black] (1, 5.5) -- (0.5, 1) -- (6, 3) -- (4, 5) -- (1, 5.5);
\node at (6.5, 6){{\scriptsize $\overline{z}^2$}};

\draw [fill=black,opacity=0.1] (1.25, 5.44) -- (4, 5) -- (4.76, 4.22) -- (2.8, 1.8) -- (1.38, 1.34) -- (1.25, 1.5) -- (1.25, 5.44);

\draw [thick, color=black, dashed] (1.25-1, 1.5+2) -- (1.25+1, 1.5-2) ;
\draw[fill,color = black] (1.25, 1.5) circle [radius=\circlerad pt];

\node[color=black,opacity=0.75] at (1, -0.5) {{\scriptsize 1}};
\node[color=black,opacity=0.75] at (-0.5, 1) {{\scriptsize 1}};


\def\shiftpicture{10}

\draw [<->,thick] (\shiftpicture+0,6) -- (\shiftpicture+0,0) -- (\shiftpicture+8,0);
\draw[help lines,opacity = 0.5] (\shiftpicture + 0,0) grid (\shiftpicture + 8,6);
\node at (\shiftpicture+8, -0.5){{\scriptsize $y_1$}};
\node at (\shiftpicture-0.5, 6){{\scriptsize $y_2$}};

\foreach \y in {1,...,5}
	\draw (\shiftpicture-0.1, \y) -- (\shiftpicture+0.1, \y);
	
\foreach \y in {1,...,7}
	\draw  (\shiftpicture+\y, -0.1) -- (\shiftpicture+\y, 0.1);

\draw [thick, color=black] (\shiftpicture + 1.25, 5.75) -- (\shiftpicture + 1.25,1.5) -- (\shiftpicture + 2, 0.75)  -- (\shiftpicture + 6, 5.75) -- (\shiftpicture + 1.25, 5.75);
\node at (\shiftpicture + 6.5, 3){{\scriptsize $\overline{z}^1$}};
\draw [thick, color=black] (\shiftpicture + 1, 5.5) -- (\shiftpicture + 0.5, 1) -- (\shiftpicture + 6, 3) -- (\shiftpicture + 4, 5) -- (\shiftpicture + 1, 5.5);
\node at (\shiftpicture + 6.5, 6){{\scriptsize $\overline{z}^2$}};

\foreach \y in {2,...,5}
	\draw[color = black] (\shiftpicture + 1, \y) circle [radius=\circlerad pt];
	
\foreach \y in {1,...,5}
	\draw[color = black] (\shiftpicture + 2, \y) circle [radius=\circlerad pt];
		
\foreach \y in {2,...,5}
	\draw[color = black] (\shiftpicture + 3, \y) circle [radius=\circlerad pt];
			
\foreach \y in {3,...,5}
	\draw[color = black] (\shiftpicture + 4, \y) circle [radius=\circlerad pt];
			
\foreach \y in {3,...,5}
	\draw[color = black] (\shiftpicture + 5, \y) circle [radius=\circlerad pt];
	
\draw[color = black] (\shiftpicture + 6, 3) circle [radius=\circlerad pt];

\draw [thick, color=black, dashed] (\shiftpicture + 2-1, 2+2) -- (\shiftpicture + 2+1, 2-2) ;
\draw[fill,color = black] (\shiftpicture + 2, 2) circle [radius=\circlerad pt];
\end{tikzpicture}
\captionsetup{width=14cm}
\caption{Illustration of the difference between detecting critical scenarios in the continuous (left) and integer case with B{\&}B (right).} \label{fig.detection}
\end{figure}
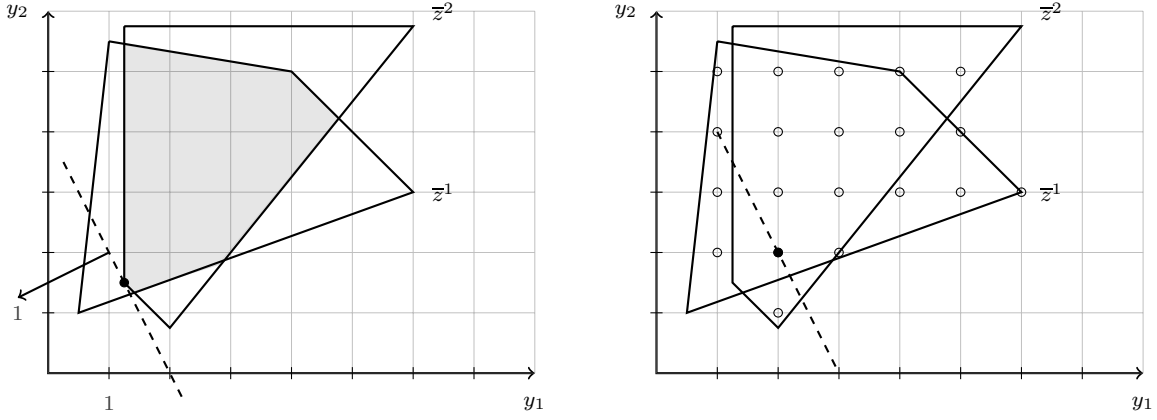
\end{example}

In this note, we propose a general approach for detecting sets $\overline{Z}_{r,s}$ of critical scenarios in adjustable RO problems with integer decision variables. In fact, we will identify these critical scenarios based on the dual solutions of the nodes in the B{\&}B tree used to solve the static mixed-integer RO problem (\ref{eq:two_period_after_general_splitting}). In Section~\ref{sec.BBsplitting} we discuss the theory behind our approach.

Throughout the remainder of this note we make the following mild assumption.
\begin{assumption} \label{assumption.nice.one}
The problem ($\text{RO}_0$) is feasible and the feasible region of its LP relaxation is non-empty and bounded.
\end{assumption}
\section{Critical scenario detection using B{\&}B} \label{sec.BBsplitting}
In this section we show how to detect critical scenarios $\overline{Z}_{r,s}$ for the static mixed-integer RO problem (\ref{eq:two_period_after_general_splitting}) after $r$ splitting rounds. Let $\overline{t}^{r,n}$ denote the objective value at each node $n \in \mathcal{N}_r$ of the B{\&}B tree used to solve (\ref{eq:two_period_after_general_splitting}). We will show that we can use the optimal dual variables at each node $n \in \mathcal{N}_r$ with $\overline{t}^{r,n} \geq \overline{t}^r$ to construct sets of critical scenarios $\overline{Z}_{r,s}$. In fact, we will show that it suffices to only consider nodes $n$ in a so-called \emph{critical cutset} $\mathcal{O}_r$ of the B{\&}B tree. This is a set of nodes that separates the root node from the leaf nodes in the B{\&}B tree, see Figure~\ref{fig.critical.nodes}. 

\begin{figure}
\centering
\begin{tikzpicture}[sibling distance=10em]
\tikzset{style1/.style={shape=rectangle, rounded corners,
    draw, align=center}}
\tikzset{style2/.style={shape=rectangle, rounded corners,
    draw, align=center, fill=gray!20}}
  \node [style1] {$n=0$, $\overline{t}^{r,0} < \overline{t}^r$}
    child { node [style2] {$n=1$, $\overline{t}^{r,1} = \overline{t}^r$} }
    child { node [style1] {$n=2$, $\overline{t}^{r,2} < \overline{t}^r$}
      child { node [style2] {$n=3$, $\overline{t}^{r,3} \geq \overline{t}^r$}
        child { node [style1] {$n=5$, $\overline{t}^{r,5} \geq \overline{t}^r$} }
        child { node [style1] {$n=6$, $\overline{t}^{r,6} \geq \overline{t}^r$} }
        child { node [style1] {$n=7$, $\overline{t}^{r,7} \geq \overline{t}^r$} } }
      child { node [style2] {$n=4$, $\overline{t}^{r,4} \geq \overline{t}^r$} } };
\end{tikzpicture}
\captionsetup{width=14cm}
\caption{An example of a critical cutset set $\mathcal{O}_r = \{ 1, 3, 4\}$ (in gray) of nodes in a B{\&}B tree.}
\label{fig.critical.nodes}
\end{figure}
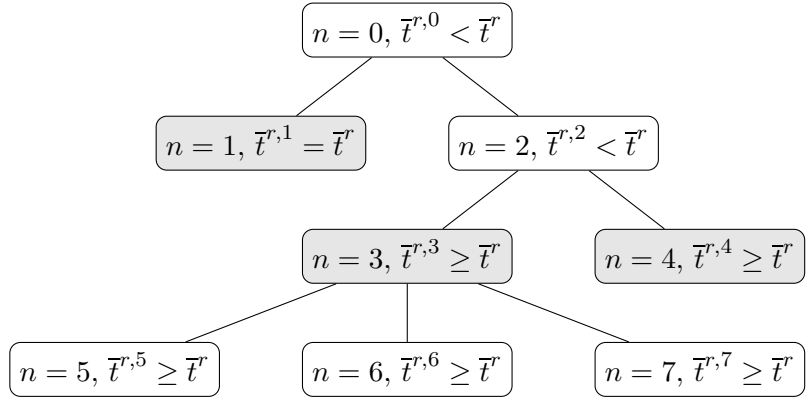

\begin{definition} \label{definition.cutset}
Let $\mathcal{N}_r$ denote the nodes of the B{\&}B tree used to solve (\ref{eq:two_period_after_general_splitting}). Then, $\mathcal{O}_r \subset \mathcal{N}_r$ is called a \emph{critical cutset} of the B{\&}B tree if
\begin{enumerate}
\item[(i)] $\overline{t}^{r,n} \geq \overline{t}^r$ for all $n \in \mathcal{O}_r$,
\item[(ii)] $\mathcal{O}_r \cap \Pi(n) \neq \emptyset$ for all $n \in \Lambda_r$,
\end{enumerate}
where $\Pi(n)$ represents the path from the root node to the leaf node $n \in \Lambda_r$ in the B{\&}B tree.  
\end{definition}
In Section \ref{sec.strong.lp.duality.node} we construct the primal and dual LPs corresponding to each node $n \in {\cal N}_r$ of the B{\&}B tree, and we show that for each node $n$ strong LP duality holds. In Section 3.2 we use the optimal dual variables of nodes $n \in \mathcal{O}_r$ in a critical cutset to construct sets of critical scenarios $\overline{Z}_{r,s}$ and we prove that if these sets are not split after round $r'$, then the worst-case objective value does not improve and thus $\overline{t}^{r'} = \overline{t}^r$ for $r' \geq r'$.
\subsection{Strong LP duality at B\&B nodes} \label{sec.strong.lp.duality.node}
At each node $n \in \mathcal{N}_r$ of the B{\&}B tree used to solve (\ref{eq:two_period_after_general_splitting}), we solve the LP relaxation of (\ref{eq:two_period_after_general_splitting}) with several additional branching constraints of the form $x_k^r \leq \delta_k^r$ or $y_l^{r,s} \geq \delta_{l}^{r,s}$. The problem that we solve at the $n$-th node of the tree equals:
\begin{align}
\overline{t}^{r,n} :=  \min\limits_{t^{r,n}, x^{r,n},y^{r,n,s}} \ & t^{r,n} \label{eq.splitted.problem.per.node} \tag{$\text{RO}_{r,n}$}\\
\text{s.t.} \ & t^{r,n} - c(z)^\top x^{r,n} - q(z)^\top y^{r,n,s} \geq 0 && \forall z \in Z_{r,s} && \forall s \in \mathcal{S}_r  \nonumber \\
& a_i(z)^\top x^{r,n} + w_i(z)^\top y^{r,n,s} \geq b_i && \forall z \in Z_{r,s}, && \forall i \in \mathcal{I}, \  \forall s \in \mathcal{S}_r  \nonumber \\
& (d^{r,n}_j)^\top x^{r,n} + \sum\limits_{s \in \mathcal{S}_r } (e^{r,n,s}_j)^\top y^{r,n,s} \geq \delta^{r,n}_j &&&& \forall  j \in \mathcal{J}_{r,n} \nonumber \\
& x^{r,n} \geq 0, \,\, y^{r,n,s} \geq 0  &&&&    \forall s \in \mathcal{S}_r, \nonumber
\end{align}
where the last constraints represent the branching constraints.
\begin{remark}
We present the branching constraints in (\ref{eq.splitted.problem.per.node}) in such a general form since our results also hold in the more general framework of disjunctive programming. However, we prefer to present our results in the context of B{\&}B. 
\end{remark} 
Problem (\ref{eq.splitted.problem.per.node}) is not a standard LP since it has infinitely many constraints. However, since $Z$ is a non-empty and bounded polyhedron by Assumption~\ref{assumption.polyhedron}, we are able to derive its robust counterpart:
\begin{align}
\overline{t}^{r,n} = \min\limits_{t^{r,n}, x^{r,n}, y^{r,n,s}, \kappa_i^{r,n,s}} \ & t^{r,n} \label{eq.initial.problem.rc1} \tag{$\text{P-RC}_{r,n}$} \\
\text{s.t.} \ & t^{r,n} - \overline{c}^\top x^{r,n} - \overline{q}^\top y^{r,n,s} - (p^{r,s})^\top \kappa_0^{r,n,s} \geq 0 && \forall s \in \mathcal{S}_r \nonumber \\
& \overline{a}_i^\top x^{r,n} + \overline{w}_i^\top y^{r,n,s} - (p^{r,s})^\top \kappa_i^{r,n,s} \geq b_i && \forall i \in \mathcal{I}, \ \forall s \in \mathcal{S}_r  \nonumber \\
& (d^{r,n}_j)^\top x^{r,n} + \sum\limits_{s \in \mathcal{S}_r } (e^{r,n,s}_j)^\top y^{r,n,s} \geq \delta^{r,n}_j && \forall  j \in \mathcal{J}_{r,n} \nonumber \\
& C^\top x^{r,n} + Q^\top y^{r,n,s} - (P^{r,s})^\top \kappa_0^{r,n,s} = 0 &&  \forall s \in \mathcal{S}_r  \nonumber \\
& A_i^\top x^{r,n} + W_i^\top y^{r,n,s} + (P^{r,s})^\top \kappa_i^{r,n,s} = 0 && \forall i \in \mathcal{I}, \ \forall s \in \mathcal{S}_r  \nonumber \\
&x^{r,n} \geq 0, \,\, y^{r,n,s} \geq 0, \,\, \kappa_i^{r,n,s} \geq 0  &&  \forall i \in \mathcal{I} \cup \{ 0 \}, \ \forall s \in \mathcal{S}_r, \nonumber
\end{align}
where $\kappa^{r,n,s}_i$, $i \in \mathcal{I} \cup \{ 0 \}$ represent additional variables required to move from robust constraints in (\ref{eq.splitted.problem.per.node}), that hold for all $z \in Z_{r,s}$ to their robust counterparts in (\ref{eq.initial.problem.rc1}). This robust counterpart is an LP since $Z_{r,s}$ is a non-empty polyhedral uncertainty set for every $s \in \mathcal{S}_r$, and its dual is given by: 
\begin{align}
\max\limits_{\lambda^{r,n,s}_i, u^{r,n,s}_i \mu^{r,n}_j} \ & \sum\limits_{s \in \mathcal{S}_r } \sum\limits_{i \in \mathcal{I}} \lambda_{i}^{r,n,s} b_i + \sum\limits_{j \in \mathcal{J}_{r,n}} \mu^{r,n}_j \delta^{r,n}_j  \label{eq.splitted.problem.per.node.dual} \tag{$\text{D-RC}_{r,n}$} \\
\text{s.t.} \ & \sum\limits_{s \in \mathcal{S}_r } \lambda^{r,n,s}_0 = 1   \nonumber \\
& \sum\limits_{s \in \mathcal{S}_r } ( \lambda^{r,n,s}_0 \overline{c} +C u^{r,n,s}_0 ) - \sum\limits_{i \in \mathcal{I}} \sum\limits_{s \in \mathcal{S}_r } ( \lambda^{r,n,s}_i \overline{a}_i + A_i u^{r,n,s}_i ) -\sum\limits_{j \in \mathcal{J}_{r,n}} \mu^{r,n}_j d^{r,n}_j \geq 0 \nonumber \\
&  \lambda^{r,n,s}_0 \overline{q} + Q u^{r,n,s}_0 - \sum\limits_{i \in \mathcal{I}} ( \lambda^{r,n,s}_i \overline{w}_i + W_i u^{r,n,s}_i ) - \sum\limits_{j \in \mathcal{J}_{r,n}} \mu^{r,n}_j e^{r,n,s}_j \geq 0, \quad \forall s \in \mathcal{S}_r    \nonumber\\
& P^{r,s} u^{r,n,s}_i \leq \lambda^{r,n,s}_i p^{r,s}, \quad \forall i \in \mathcal{I} \cup \{ 0 \},\qquad \qquad \qquad \forall s \in \mathcal{S}_r   \nonumber \\
& \lambda^{r,n,s}_i \geq 0, \,\, \mu^{r,n,s}_j \geq 0, \quad \forall i \in \mathcal{I} \cup \{ 0 \},\qquad \qquad \qquad \forall  j \in \mathcal{J}_{r,n}, \quad \forall s \in \mathcal{S}_r. \nonumber
\end{align}
\begin{proposition}
Under Assumptions \ref{assumption.nonempty.bounded}-\ref{assumption.nice.one}, strong LP duality holds between (\ref{eq.initial.problem.rc1}) and (\ref{eq.splitted.problem.per.node.dual}) for each node $n \in \mathcal{N}_r$ of the B{\&}B tree after splitting round $r$. \label{prop.stong.LP.duality}
\end{proposition}
\begin{proof}{Proof}
Problems (\ref{eq.initial.problem.rc1}) and (\ref{eq.splitted.problem.per.node.dual}) form a standard primal dual pair. From LP duality theory (see, \emph{e.g.}, \cite{Schrijver1986}), it follows that strong LP duality holds unless both the primal and dual problem are infeasible. Thus to prove the claim it suffices to show that either the primal or dual is feasible.

Consider the static mixed-integer RO problem (\ref{eq:two_period_after_general_splitting}) with $r = 0$. By Assumption~\ref{assumption.nice.one} this problem is feasible, and the feasible region of its LP-relaxation is non-empty and bounded. The LP relaxation can be interpreted as ($\text{RO}_{r,n}$) with $r = 0$ and $n = 0$ the root node of the B{\&}B tree used to solve ($\text{RO}_{0}$). Hence, under Assumption~\ref{assumption.nice.one}, ($\text{P-RC}_{0,0}$) has a non-empty, bounded feasible region and thus a finite objective value. By strong LP duality, the objective value of ($\text{D-RC}_{0,0}$) is also finite and its feasible region thus non-empty.

Using the same arguments as above, (\ref{eq.splitted.problem.per.node.dual}) is feasible at the root node $n = 0$ for any splitting round $r$, since (\ref{eq:two_period_after_general_splitting}) is feasible because the feasible solution for ($\text{RO}_{0}$) can be implemented for (\ref{eq:two_period_after_general_splitting}) after $r$ splitting rounds using the same $y$-values for all uncertainty subsets $s \in \mathcal{S}_r$ as in ($\text{RO}_{0}$). Moreover, the additional branching constraints in (\ref{eq.initial.problem.rc1}) for arbitrary $n$ restrict the primal feasible region, but enlarge the dual feasible region. Hence, (\ref{eq.splitted.problem.per.node.dual}) is feasible for any $r$ and $n$, and thus strong LP duality between (\ref{eq.initial.problem.rc1}) and (\ref{eq.splitted.problem.per.node.dual}) always holds. 
\end{proof}

\subsection{Critical scenarios}
Next, we discuss how to obtain critical scenarios from the dual variables of (\ref{eq.splitted.problem.per.node.dual}). Recall that  $Z_{r,s} = \{z: P^{r,s}z \leq p^{r,s}\}$ and that the optimal dual variables $(\overline{\lambda},\overline{u},\overline{\mu})$ of (\ref{eq.splitted.problem.per.node.dual}) satisfy $P^{r,s}\overline{u}_i^{r,n,s} \leq \overline{\lambda}_i^{r,n,s}p^{r,s}$ for $i \in I \cup \{0\}$. Hence, if $\overline{\lambda}_i^{r,n,s} > 0$, then
$$
P^{r,s} (\overline{u}^{r,n,s}_i / \overline{\lambda}_i^{r,n,s}) \leq p^{r,s} \ \Rightarrow \ (\overline{u}^{r,n,s}_i / \overline{\lambda}_i^{r,n,s}) \in Z_{r,s}.
$$
That is, the quotient $\overline{u}_i^{r,n,s}/\overline{\lambda}_i^{r,n,s}$ can be interpreted as a scenario from the uncertainty set $Z_{r,s}$. The set of all $\overline{u}^{r,n,s}_i / \overline{\lambda}_i^{r,n,s}$ for which $\overline{\lambda}_i^{r,n,s} > 0$, $i \in \mathcal{I} \cup \{ 0 \}$, will represent the set of critical scenarios $\overline{Z}_{r,n,s}$ in node $n$ corresponding to the uncertainty subset $Z_{r,s}$. However, we need to take into account the possibility that problem (\ref{eq.initial.problem.rc1}) is infeasible and problem (\ref{eq.splitted.problem.per.node.dual}) is unbounded, and thus no optimal dual solution exists. For this reason, we call any solution $(\overline{\lambda}, \overline{u}, \overline{\mu})$ optimal if its corresponding objective value in (\ref{eq.splitted.problem.per.node.dual}) exceeds $\overline{t}^{r}$.
\begin{definition}
For every node $n \in \mathcal{O}_r$ in a critical cutset of the nodes of the B{\&}B tree used for solving the static mixed-integer RO problem (\ref{eq:two_period_after_general_splitting}), we call $(\overline{\lambda},\overline{u},\overline{\mu})$ an optimal solution of (\ref{eq.splitted.problem.per.node.dual}) if $(\overline{\lambda},\overline{u},\overline{\mu})$ is feasible and its objective value exceeds $\overline{t}^r$.
\end{definition}
\begin{definition}
Let the splitting round $r$ be given and let $\mathcal{O}_r \subset \mathcal{N}_r$ be a critical cutset of the nodes of the B{\&}B tree used to solve the static mixed-integer RO problem (\ref{eq:two_period_after_general_splitting}). Then, for each $n \in \mathcal{O}_r$ and $s \in \mathcal{S}_r$, the set of critical scenarios $\overline{Z}_{r,s,n}$ corresponding to uncertainty subset $Z_{r,s}$ in node $n$ is given by
\begin{align*}
\overline{Z}_{r,s,n} = \{ \overline{u}^{r,n,s}_i / \overline{z}^{r,n,s}_i: \ \overline{\lambda}_i^{r,n,s} > 0, \ i \in \mathcal{I} \cup \{ 0 \}  \}.
\end{align*}
Moreover, the set of \emph{critical scenarios} $\overline{Z}_{r,s}$ corresponding to the uncertainty subset $Z_{r,s}, s \in \mathcal{S}_r$, is 
$$
\overline{Z}_{r,s} = \bigcup_{n \in \mathcal{O}_r} \overline{Z}_{r,s,n}.
$$ \label{def.critical.scenarios}
\end{definition}
Now we are ready to prove our main theorem, which can be interpreted as the integer analogue of Theorem 1 in \cite{Postek2016}.
\begin{theorem} \label{thm.splitting.theorem}
Consider the static mixed-integer RO problem (\ref{eq:two_period_after_general_splitting}) under Assumptions~\ref{assumption.nonempty.bounded}-\ref{assumption.nice.one}, and assume that we solve this problem using B{\&}B. Let $\mathcal{O}_r \subset \mathcal{N}_r$ denote a critical cutset of the nodes of the B{\&}B tree used to solve (\ref{eq:two_period_after_general_splitting}). Then, for any refinement ${\cal Z}_{r'}$ of ${\cal Z}_r$ for which for every $s \in \mathcal{S}_r $, there exists $s' \in \mathcal{S}_{r'}$ such that 	
\begin{align}
	\bigcup\limits_{n \in \mathcal{O}_r} \overline{Z}_{r,n,s} \subseteq Z_{r',s'}, \ \label{eq.scenario.inclusion}
\end{align}	
we have  $\overline{t}^{r'} = \overline{t}^r$.  That is, the objective function value does not improve. %
\end{theorem}
\begin{proof}{Proof}
Assume w.l.o.g.\ that $\mathcal{S}_r \subset \mathcal{S}_{r'}$ and that the splitted sets in ${\cal Z}_{r'}$ are indexed such that $\overline{Z}_{r,n,s} \subseteq Z_{r',s}$ for all $n \in {\cal O}_r$, $s \in \mathcal{S}_r$. Since ${\cal Z}_{r'}$ is a refinement of ${\cal Z}_r$ it follows immediately that $\overline{t}^{r'} \leq \overline{t}^r$. It remains to show that also $\overline{t}^{r'} \geq \overline{t}^r$ holds. We will do so by proving that $\overline{t}^r$ is a lower bound for ($\text{RO}_{r'}$), the robust optimization problem after $r'$ rounds of splitting, using the B\&B tree of the $r$-th splitting round. Indeed, for each node $n \in \mathcal{O}_r$ we can consider ($\text{RO}_{r'}$) with the additional branching constraints from node $n \in \mathcal{O}_r$ (and without integrality restrictions):
\begin{align}
	\overline{t}^{r',n} := \min\limits_{t^{r',n}, x^{r',n}, y^{r',n,s}} \ & t^{r',n} \label{eq.rplus1.problem.primal} \tag{$\text{RO}_{r',n}$} \\
	\text{s.t.} \ & t^{r',n} - c(z)^\top x^{r',n} - q(z)^\top y^{r',n,s} \geq 0 && \forall z \in Z_{r',s} && \forall s \in \mathcal{S}_{r'} \nonumber \\
	& a_i(z)^\top x^{r',n} + w_i(z)^\top y^{r',n,s} \geq b_i  && \forall z \in Z_{r',s} && \forall i \in \mathcal{I}, \ \forall s \in \mathcal{S}_{r'} \nonumber \\
	& (d^{r,n}_j)^\top x^{r',n} +  \sum\limits_{s \in \mathcal{S}_r } (e^{r,n,s}_j)^\top y^{r',n,s} \geq \delta^{r',n}_j &&&& \forall  j \in \mathcal{J}_{r,n} \label{eq:sum_to_S_r} \\
	& x^{r,n} \geq 0, \,\, y^{r,n,s} \geq 0  &&&&    \forall s \in \mathcal{S}_{r'}. \nonumber
\end{align}	
Observe that the summation in (\ref{eq:sum_to_S_r}) runs over $s \in \mathcal{S}_r $, which means that branching conditions are only added to decision variables that were also present in round $r$. Moreover, since $\mathcal{O}_r$ is a critical cutset satisfying (i) and (ii) in Definition~\ref{definition.cutset}, it follows that the optimal solution to the problem ($\text{RO}_{r'}$) is feasible for at least one node $n \in \mathcal{O}_r$, and thus the minimum objective value of (\ref{eq.rplus1.problem.primal}) over all nodes $n \in \mathcal{O}_r$ yields a lower bound for $\overline{t}^{r'}$:
$$
\min_{n \in \mathcal{O}_r} \overline{t}^{r',n} \leq \overline{t}^{r'}.
$$	
Next, we will use the dual problems of (\ref{eq.splitted.problem.per.node}) and (\ref{eq.rplus1.problem.primal}) and the fact that $\overline{t}^{r,n} \geq \overline{t}^r$ for all $n \in \mathcal{O}_r$ to prove that $\overline{t}^{r} \leq \overline{t}^{r',n}$ for every $n \in \mathcal{O}_r$, and thus
\begin{equation}
\overline{t}^r \leq \min_{n \in \mathcal{O}_r} \overline{t}^{r',n} \leq \overline{t}^{r'}. \label{eq:min_t}
\end{equation}	
After obtaining the robust counterpart (\ref{eq.initial.problem.rc1}) of (\ref{eq.splitted.problem.per.node}), its dual is given by (\ref{eq.splitted.problem.per.node.dual}) and by Proposition~\ref{prop.stong.LP.duality}, strong LP duality holds between the two. Similarly, the dual of (\ref{eq.rplus1.problem.primal}) is equivalent to:
\begin{align}	
\max\limits_{\lambda^{r',n,s}_i, u^{r',n,s}_i, \mu^{r',n}_j} \ & \sum\limits_{s\in \mathcal{S}_{r'}} \sum\limits_{i \in \mathcal{I}} \lambda_{i}^{r',n,s} b_i + \sum\limits_{j \in \mathcal{J}_{r,n}} \mu^{r',n}_j \delta^{r,n}_j  \label{eq.rplus1.problem.dual}  \tag{$\text{D-RC}_{r',n}$}\\
\text{s.t.} \ & \sum\limits_{s\in \mathcal{S}_{r'}} \lambda^{r',n,s}_0 = 1 &&  \nonumber \\
& \sum\limits_{s\in \mathcal{S}_{r'}} (\lambda^{r',n,s}_0 \overline{c} + C u^{r',n,s}_0 ) - \sum\limits_{i \in \mathcal{I}} \sum\limits_{s\in \mathcal{S}_{r'}} (\lambda^{r',n,s}_i \overline{a}_i + A_i u^{r',n,s}_i )  \nonumber \\
\qquad \qquad & - \sum\limits_{  j \in \mathcal{J}_{r,n}} \mu^{r',n}_j d^{r,n}_j \geq 0 \nonumber \\
& \lambda^{r',n,s}_0 \overline{q} + Q u^{r',n,s}_0 - \sum\limits_{i \in \mathcal{I}} (\lambda^{r',n,s}_i \overline{w}_i + W_i u^{r',n,s}_i ) \nonumber \\ 
& \qquad \qquad - \sum\limits_{j \in \mathcal{J}_{r,n}} \mu^{r',n}_j e^{r,n,s}_j \geq 0, \quad \forall s \in \mathcal{S}_r  \nonumber\\
& \lambda^{r',n,s}_0 \overline{q} + Q u^{r',n,s}_0 - \sum\limits_{i \in \mathcal{I}} ( \lambda^{r',n,s}_i \overline{w}_i + W_i u^{r',n,s}_i ) \geq 0, \qquad \forall s \in \mathcal{S}_{r'} \setminus  \mathcal{S}_r    \nonumber\\
& P^{r',s} u^{r',n,s}_i \leq \lambda^{r',n,s}_i p^{r',s}, \quad \forall i \in \mathcal{I} \cup \{ 0 \}, \qquad \qquad  \qquad \qquad \forall s \in \mathcal{S}_{r'}  \label{eq:Prs} \\
& \lambda^{r',n,s}_i \geq 0,\,\, \mu^{r',n}_j \geq 0, \quad \forall i \in \mathcal{I} \cup \{ 0 \}, \qquad \qquad \qquad \forall j \in \mathcal{J}_{r,n}, \quad \forall s \in \mathcal{S}_{r'}. \nonumber
\end{align}	
Observe that it	is possible to select $u^{r',n,s} = 0$ and $\lambda^{r',s} = 0$ for all $s \in \mathcal{S}_{r'} \setminus \mathcal{S}_{r}$ to obtain the same dual as in (\ref{eq.splitted.problem.per.node.dual}) except that $P^{r',s}$ and $p^{r',s}$ in the constraints in (\ref{eq:Prs}) refer to the splitted uncertainty sets in $\mathcal{Z}_{r'}$, whereas in (\ref{eq.splitted.problem.per.node.dual}) they refer to the uncertainty sets in ${\cal Z}_r$. These constraints, however, can be written in a different form, since the uncertainty sets $Z_{r',s}$ are bounded and thus $P^{r',s}z \leq 0 \Rightarrow z = 0$. Hence, if $\lambda_i^{r',n,s} = 0$, then $u_i^{r',n,s} = 0$, and if $\lambda_i^{r',n,s} > 0$, then the constraint reduces to
$$
P^{r',s}\left(\frac{u_i^{r',n,s}}{\lambda_i^{r',n,s}}\right) \leq p^{r',s} \quad \Leftrightarrow \quad \frac{u_i^{r',n,s}}{\lambda_i^{r',n,s}} \in Z^{r',s}.
$$
Using this alternative form, and since  $\bigcup\limits_{n \in \mathcal{O}_r} \overline{Z}_{r,n,s} \subseteq Z_{r',s}$ for all $s \in \mathcal{S}_r $, it is not hard to verify that the optimal dual solutions of (\ref{eq.splitted.problem.per.node.dual}) can be used to construct a feasible solution to (\ref{eq.rplus1.problem.dual}) for every $ i \in \mathcal{I} \cup \{ 0 \}$ and  $j \in \mathcal{J}_{r,n}$:
\begin{align*}
	\lambda_i^{r',n,s} & := \left\{ \begin{array}{ll} \overline{\lambda}_i^{r,n,s} & \text{if } s \in \mathcal{S}_r\ \text{and}\ \overline{\lambda}_i^{r,n,s} > 0, \\ 0 & \text{otherwise,} \end{array} \right. \\
	u_i^{r',n,s} & := \left\{ \begin{array}{ll} \overline{u}_i^{r,n,s} & \text{if } s \in \mathcal{S}_r \ \text{and} \ \overline{\lambda}_i^{r,n,s} > 0 \\ 0 & \text{otherwise,} \end{array} \right. \\
	\mu^{r',n}_j & := \overline{\mu}^{r,n}_j.
\end{align*}
The objective value of this dual solution is at least $\overline{t}^r$ by definition of a critical cutset. Thus, for all $n \in \mathcal{O}_r$, we have by weak LP duality that $\overline{t}^{r',n} \geq \overline{t}^r$, and thus by (\ref{eq:min_t}) we conclude that $\overline{t}^r \leq \overline{t}^{r'}$. 
\end{proof}
Theorem~\ref{thm.splitting.theorem} shows how to detect critical scenarios to be split for mixed-integer adjustable RO problems in general. In Section~\ref{sec.experiment} we will show that this critical scenario detection method outperforms existing problem-specific heuristic methods. First, however, we remark that Theorem 1 implies a simple optimality criterion for when to stop splitting the uncertainty set.
\begin{corollary}
Let $\mathcal{O}_r \subset \mathcal{N}_r$ denote a critical cutset of the nodes of the B{\&}B tree used for solving (\ref{eq:two_period_after_general_splitting}), and suppose that
\begin{align*}
\left| \bigcup\limits_{n \in \mathcal{O}_r} \overline{Z}_{r,n,s} \right| \leq 1, \ \forall s \in \mathcal{S}_r.
\end{align*}
Then, for any refinement $Z_{r'}$ of $Z_r$ we have $\overline{t}^{r'} = \overline{t}^r$, and thus the objective value of the adjustable RO problem (\ref{eq.initial.problem}) equals $\overline{t}^r$. 
\end{corollary}
\begin{remark}
Due to the strong LP duality between problems (\ref{eq.initial.problem.rc1}) and (\ref{eq.splitted.problem.per.node.dual}), the critical scenarios can be obtained at no cost from the optimal dual multipliers of the LP problem if (\ref{eq.initial.problem.rc1}) is feasible. In a similar way, they can be obtained from the dual infeasibility ray of problem (\ref{eq.initial.problem.rc1}) if (\ref{eq.initial.problem.rc1}) is infeasible and (\ref{eq.splitted.problem.per.node.dual}) is unbounded. In this way, no additional optimization problems need to be solved to construct the set of critical scenarios.
\end{remark}
\section{Numerical experiment - route planning} \label{sec.experiment}
\subsection{Problem description}
To illustrate the potential benefits of our method we consider the route planning problem from \cite{Hanasusanto2015} and \cite{Postek2016}. This is a problem with uncertainty in the objective function to which the methodology of \cite{Postek2016} and \cite{Bertsimas2016} cannot be straightforwardly applied.

The problem is a shortest path problem defined on a directed graph $G = (V,A)$ with nodes $V = \{1,\ldots,N\}$, arcs $A \subseteq V \times V$, and uncertain weights $w_{ij}(z) \in \mathbb{R}_{+}$ for every arc $(i,j) \in A$. We assume that these arc weights are affine functions of the uncertain parameters $z \in Z$, where $Z$ is a polyhedral uncertainty set. The goal is to determine the length of the worst-case shortest path from a start node $b \in V$ to an end node $e \in V$ with $b \neq e$. This shortest path is determined after we observe the arc weights $w_{ij}(z)$, but its its worst-case length is determined before these arc weights are known. If the arc weights represent travel times, then this problem can be interpreted as a route planning problem in which we determine the maximum time required to travel from node $b$ to $e$.

The corresponding mixed-integer adjustable RO problem is given by:
\begin{align}
\label{eq:Route_planning} \tag{RPP}
\min\limits_{t, y(z)} \ & t \\
\text{s.t.} \ & t - \sum\limits_{(i,j) \in A} w_{ij}(z) y_{ij}(z)\geq 0, && \forall z \in Z \nonumber \\ 
& \sum\limits_{(j,l) \in A} y_{jl}(z) - \sum\limits_{(i,j) \in A } y_{ij}(z) \geq \mathbb{I}(j=b) - \mathbb{I}(j=e), \quad && \forall j \in V  \nonumber \\
& y_{ij}(z) \in \{0,1\} , && \forall z \in Z, \  \forall (i,j) \in A. \nonumber
\end{align}
Here, the binary variables $y_{ij}(z)$ are equal to 1 if arc $(i,j)$ is part of the shortest path from $b$ to $e$, and $\mathbb{I}(\cdot)$ denotes the indicator function.

Since this problem has uncertainty in the objective function only, the methodology of \cite{Postek2016} only generates a single critical scenario, giving no indication on how to split the uncertainty set. That is why they propose a problem-specific heuristic splitting rules that can only be applied to this route planning problem. In our numerical experiment we show that our general B\&B-based critical scenario detection method outperforms these splitting rules.
\subsection{Experimental design}
We generate random graphs with $N$ nodes, where the location of each node is uniformly sampled from $[0,10]^2$. The nodes between which the Euclidean distance is largest are designated as start and end node. Moreover, the arc set $A$ is obtained by removing the longest 70\% of arcs from a complete directed graph, as in \cite{Hanasusanto2015} and \cite{Postek2016}.

We assume that the arc weights $w_{ij}(z)$ are defined as
\[w_{ij}(z) = (1+z_{ij}/2)d_{ij},\]
where $d_{ij}$ represents the Euclidean distance between nodes $i$ and $j$, and $z$ is contained in the polyhedral uncertainty set
\[Z = \bigg\{z \in [0,1]^{|A|}: \sum_{(i,j) \in A} z_{ij} \leq B\bigg\}. \]
Thus, the arc weight $w_{ij}(z)$ may be between 100\% and 150\% of the distance $d_{ij}$ between the nodes. For the parameter $B$ in the uncertainty set we consider $B = 2, 3, 4$.

In our numerical experiment we compare the quality of the uncertainty set splits based on (i) the problem-specific method of \cite{Postek2016} and (ii) our new B{\&}B-based critical scenario detection method. For both methods, we split the uncertainty sets $Z_{r,s}$ for those $s \in \mathcal{S}_r$ for which the first constraint in (\ref{eq:two_period_after_general_splitting}) is active. That is,  we only focus on those uncertainty subsets that determine the worst-case objective value after the $r$-th splitting round. Each such set is split into two subsets using the bisecant plane between the two critical scenarios that are furthest apart from each other. A bisecant plane between two points $z$ and $z'$ is the hyperplane going through the point $(z + z') / 2$ with normal vector $z - z'$.

The idea behind the heuristic of \cite{Postek2016} is to find an alternative critical scenario $z$, so that $z$ and $z_{LP}$, obtained using the LP-relaxation, can be split. To make sure that $z$ differs substantially from $z_{LP}$ its optimal path $y(z)$ cannot use more than $100\theta\%$ of the arcs in the optimal path $y(z_{LP})$ corresponding to $z_{LP}$. Here, $0 \leq \theta \leq 1$ is a parameter that we can select. For details of this \emph{ad-hoc} method, see \cite{Postek2016}. Our new general critical scenario detection method identifies critical scenarios as explained in Definition~\ref{def.critical.scenarios}. In our numerical experiments we use the critical cutset $\mathcal{O}_r$ with smallest cardinality. 


\begin{remark}[Ex post correction] \label{remark.ex.post}
Since our route planning problem (\ref{eq:Route_planning}) only has uncertainty in the objective function, it is possible to apply an \emph{ex post correction} to the worst-case objective value after each round of splitting. The idea, not recognized in \cite{Postek2016}, is that the routes $y^{r,s}$ corresponding to uncertainty sets $Z_{r,s}$,  $s \in \mathcal{S}_r$, after $r$ rounds of splitting are feasible for any $z \in Z$. By selecting the best route among $y^{r,s}, s \in \mathcal{S}_r$, for every $z \in Z$, the worst-case objective value becomes
$$
\underline{t}_r := \max_{z \in Z} \min_{s \in \mathcal{S}_r} \sum_{(i,j) \in A} w_{ij}(z)^{\top}y_{ij}^{r,s}.
$$
This objective value $\underline{t}_r$ may be lower than $\bar{t}_r$ since $y^{r,s'}$ is not necessarily the best solution among $y^{r,s},  s \in \mathcal{S}_r$, for all $z \in Z_{r,s'}$. In our numerical experiment we apply this ex-post correction and show both values $\underline{t}_r$ and $\bar{t}_r$.
\end{remark}
\subsection{Results}
In Table \ref{tab.numerical.results.modified} we present results for a representative parameter set $N = 10, 20, 30, 40$, $\theta = 0, 0.5, 0.9$ and $B = 3$. For each value of $N$, the worst-case objective value improvement of $\text{RO}_r$ compared to the objective value of $\text{RO}_0$  is given, both for our B{\&}B scenario detection method and the problem-specific method of \cite{Postek2016} with $\theta = 0, 0.5, 0.9$. We report the worst-case objective value improvement after a single splitting round, \emph{i.e.}, when the uncertainty set is partitioned into two subsets and thus $|\mathcal{Z}_r| = 2$, and when $|\mathcal{Z}_r| = 10$.

\begin{table}
\centering
\scriptsize
\captionsetup{width=14cm}
\caption{Improvements (\%) in the worst-case objective function value. In bold result of the approach that performed best for a given  $N$ if its average outperformance was statistically significant at 0.95 confidence level.} \label{tab.numerical.results.modified}
\centerline{\begin{tabular}{l|cccc|cccc|cccc|cccc} \hline
& \multicolumn{4}{c|}{$N = 10$} & \multicolumn{4}{c|}{$N = 20$} & \multicolumn{4}{c|}{$N = 30$} & \multicolumn{4}{c}{$N = 40$} \\
$|\mathcal{Z}_r | = 2$ & B{\&}B & \multicolumn{3}{c|}{PdH (2016) $\theta$} & B{\&}B & \multicolumn{3}{c|}{PdH (2016) $\theta$}& B{\&}B & \multicolumn{3}{c|}{PdH (2016) $\theta$}& B{\&}B & \multicolumn{3}{c}{PdH (2016) $\theta$} \\ \hline
Correction & & $0$ & $0.5$ & $0.9$ & & $0$ & $0.5$ & $0.9$ & & $0$ & $0.5$ & $0.9$ & & $0$ & $0.5$ & $0.9$ \\ \hline
No & 1.10 & 0.24 & 0.82 & 0.87 & 2.84 & 2.45 & 2.34 & 0.98& 3.42 & 3.88 & 3.13 & 1.01& 2.95 & 4.22 & 2.34 & 0.90 \\ 
Ex post & 2.13 & 0.76 & 1.51 & 1.52 & 5.44 & 5.27 & 4.68 & 1.79& 6.16 & 6.75 & 5.09 & 1.91& 5.62 & 7.50 & 3.98 & 1.65 \\ \hline
& \multicolumn{4}{c|}{$N = 10$} & \multicolumn{4}{c|}{$N = 20$} & \multicolumn{4}{c|}{$N = 30$} & \multicolumn{4}{c}{$N = 40$} \\
$|\mathcal{Z}_r| = 10$ & B{\&}B & \multicolumn{3}{c|}{PdH (2016) $\theta$} & B{\&}B & \multicolumn{3}{c|}{PdH (2016) $\theta$}& B{\&}B & \multicolumn{3}{c|}{PdH (2016) $\theta$}& B{\&}B & \multicolumn{3}{c}{PdH (2016) $\theta$} \\ \hline
Correction & &  $0$ & $0.5$ & $0.9$ & & $0$ & $0.5$ & $0.9$ & & $0$ & $0.5$ & $0.9$ & & $0$ & $0.5$ & $0.9$ \\ \hline
No & \textbf{2.36} & 1.79 & 2.20 & 1.79 & 4.42 & 4.16 & 3.85 & 1.53& 5.30 & 5.49 & 4.61 & 1.88& 5.00 & \textbf{5.70} & 4.18 & 1.39 \\ 
Ex post & \textbf{3.22} & 3.02 & 3.14 & 2.34 & \textbf{8.41} & 7.21 & 6.36 & 2.91& \textbf{9.91} & 8.50 & 6.91 & 3.82& \textbf{11.27} & 8.60 & 7.56 & 3.34 \\ \hline 
\end{tabular}}
\end{table}

We observed that the results are very similar for $B = 2, 3, 4$, both for $|\mathcal{Z}_r| = 2$ and $|\mathcal{Z}_r| = 10$. The performance of the problem-specific heuristic, however, depends strongly on the parameter $\theta$, and is better for $\theta = 0$ than for $\theta = 0.5$ and $\theta = 0.9$. Moreover, the \emph{ex post} correction of the worst-case objective value discussed in Remark~\ref{remark.ex.post} also has a substantial impact. In all cases it leads to a major increase in objective value improvement. For example, for $N = 40$ and $B = 3$ the improvement of our B{\&}B scenario detection method is $5.00\%$ without and $11.27\%$ with \emph{ex post} correction.

Comparing the B{\&}B scenario detection method with the problem-specific method of \cite{Postek2016}, using the \emph{ex post} corrected results, the B{\&}B method outperforms the problem-specific method. In fact, for all $N$ the worst-case objective value improvement of the B{\&}B method is statistically significantly better than the problem-specific method for all values of $\theta$. This results is confirmed in Figure~\ref{eq.plot.improvements} where we show the worst-case objective improvements of both methods as a function of $| \mathcal{Z}_r |$, the number of subsets in which $Z$ is partitioned, for $N = 30$ and $B = 3$. Observe that the largest objective value improvements are from the initial splits of the uncertainty set. Moreover, the increase in the objective value improvement diminishes with the number of splits.

\begin{figure}
\centering
\begin{tikzpicture}[scale = 0.8]

\begin{axis}[%
width=1.787in,
height=1.642in,
at={(1.089in,0.455in)},
scale only axis,
xmin=1,
xmax=10,
xlabel style={font=\color{white!15!black}},
xlabel={$|\mathcal{Z}_r|$},
ymin=0,
ymax=10,
ylabel style={font=\color{white!15!black}},
ylabel={Obj improvement (\%)},
axis background/.style={fill=white},
title style={font=\bfseries},
title={$\theta = 0$},
legend style={at={(0.97,0.03)}, anchor=south east, legend cell align=left, align=left, draw=white!15!black}
]
\addplot [color=black, line width=1.0pt]
  table[row sep=crcr]{%
1	0\\
2	6.15616517193306\\
3	7.25731084930356\\
4	8.13828863319622\\
5	8.79501535933495\\
6	9.15463417906602\\
7	9.43881409617062\\
8	9.5226563908794\\
9	9.62801921080834\\
10	9.90630851989406\\
};
\addlegendentry{B\&B}

\addplot [color=black, dashed, line width=1.0pt]
  table[row sep=crcr]{%
1	0\\
2	6.75064855604706\\
3	7.44107299190904\\
4	7.80660401459052\\
5	7.9291720154606\\
6	8.0586088425724\\
7	8.11094634517857\\
8	8.3491425509093\\
9	8.39651413193023\\
10	8.49918331060065\\
};
\addlegendentry{Ad hoc}

\end{axis}

\begin{axis}[%
width=1.787in,
height=1.642in,
at={(3.44in,0.455in)},
scale only axis,
xmin=1,
xmax=10,
xlabel style={font=\color{white!15!black}},
xlabel={$|\mathcal{Z}_r|$},
ymin=0,
ymax=10,
axis background/.style={fill=white},
title style={font=\bfseries},
title={$\theta = 0.5$}
]
\addplot [color=black, line width=1.0pt, forget plot]
  table[row sep=crcr]{%
1	0\\
2	6.15616517193306\\
3	7.25731084930356\\
4	8.13828863319622\\
5	8.79501535933495\\
6	9.15463417906602\\
7	9.43881409617062\\
8	9.5226563908794\\
9	9.62801921080834\\
10	9.90630851989406\\
};
\addplot [color=black, dashed, line width=1.0pt, forget plot]
  table[row sep=crcr]{%
1	0\\
2	5.08650820137832\\
3	5.43318602603986\\
4	5.89142408246632\\
5	6.06669942582505\\
6	6.34189934934264\\
7	6.40275371431634\\
8	6.61023716202623\\
9	6.65762708552831\\
10	6.91358996103604\\
};
\end{axis}

\begin{axis}[%
width=1.787in,
height=1.642in,
at={(5.792in,0.455in)},
scale only axis,
xmin=1,
xmax=10,
xlabel style={font=\color{white!15!black}},
xlabel={$|\mathcal{Z}_r|$},
ymin=0,
ymax=10,
axis background/.style={fill=white},
title style={font=\bfseries},
title={$\theta = 0.9$}
]
\addplot [color=black, line width=1.0pt, forget plot]
  table[row sep=crcr]{%
1	0\\
2	6.15616517193306\\
3	7.25731084930356\\
4	8.13828863319622\\
5	8.79501535933495\\
6	9.15463417906602\\
7	9.43881409617062\\
8	9.5226563908794\\
9	9.62801921080834\\
10	9.90630851989406\\
};
\addplot [color=black, dashed, line width=1.0pt, forget plot]
  table[row sep=crcr]{%
1	0\\
2	1.90630427263504\\
3	2.58439779869521\\
4	2.90338817606778\\
5	3.01914779109442\\
6	3.31452719453483\\
7	3.342356065033\\
8	3.63148227405021\\
9	3.71859262382208\\
10	3.82154931187745\\
};
\end{axis}
\end{tikzpicture}%
\captionsetup{width=14cm}
\caption{Objective function improvement (\%) of the B{\&}B and the problem-specific methods relative to the number of subsets over the splitting process (with the \textit{ex post} correction), for $N = 30$, $B = 3$.}
 \label{eq.plot.improvements}
\end{figure}
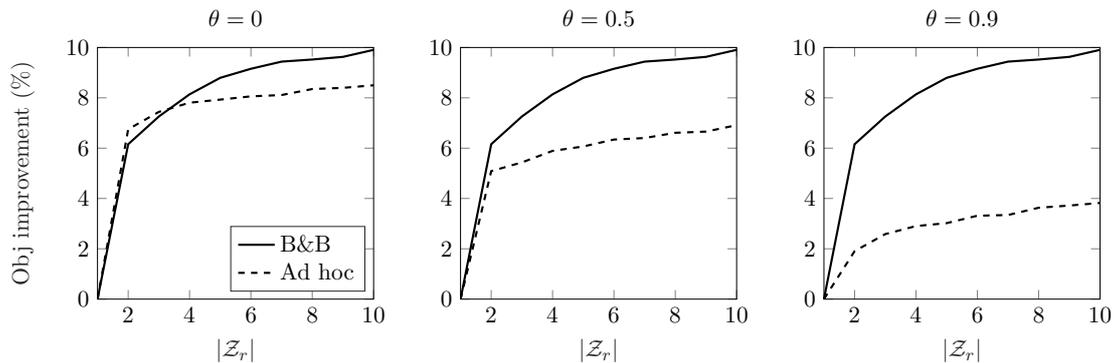

\begin{remark}
The running times of the problem-specific heuristic are lower than those of our B{\&}B scenario detection method. We do not report these running times since we have used a self-implemented B{\&}B framework to extract optimal dual vectors from nodes in the B{\&}B tree. To our knowledge, this is not possible using current commercial solvers.
\end{remark}

\section{Summary} \label{sec.Summary}
In this note, we have considered piecewise constant decision rules for mixed-integer adjustable robust optimization (RO) by adaptively partitioning the uncertainty set, as proposed by \cite{Postek2016} and \cite{Bertsimas2016}. In this approach, the uncertainty set is iteratively partitioned into smaller subsets in such a way that so-called critical scenarios are located in separate subsets. An open issue in this context has been how to detect these critical scenarios in problems involving integer decision variables. That is why we have provided a general-purpose critical scenario detection method for such problems that is based on the B{\&}B tree used to solve the corresponding static mixed-integer RO problem. In particular, the critical scenarios are directly derived from the optimal dual vectors in the nodes of the B{\&}B tree, at no extra computational cost.
Using numerical experiments on a route planning problem, we have shown that our general-purpose method outperforms the problem-specific heuristic method of \cite{Postek2016}.

\end{document}